\documentclass[12pt,twoside]{amsart}
\usepackage{amssymb,amsthm,amsmath,amsfonts,amscd,bm,enumerate}
\newtheorem{theo}{Theorem}[section]
\newtheorem{prop}[theo]{Proposition}
\newtheorem{lem}[theo]{Lemma}
\newtheorem*{claim}{Claim}
\newtheorem{conj}[theo]{Conjecture}
\newtheorem{cor}[theo]{Corollary}

\theoremstyle{definition}
\newtheorem{defi}[theo]{Definition}
\newtheorem{exmp}[theo]{Example}
\newtheorem{rem}[theo]{Remark}
\newtheorem{Step}{Step}

\newcommand{\spec}{\mathop{\mathrm{Spec}}\nolimits}
\newcommand{\id}{\mathrm{id}}

\newcommand{\Fpt}{\mathrm{Fpt}}
\newcommand{\lct}{\mathop{\mathrm{lct}}}
\newcommand{\ch}{\mathrm{char}}
\newcommand{\diag}{\mathrm{diag}}
\newcommand{\wt}{\mathrm{wt}}
\newcommand{\sing}{\mathop{\mathrm{Sing}}\nolimits}
\newcommand{\supp}{\mathop{\mathrm{Supp}}\nolimits}
\newcommand{\divis}{\mathop{\mathrm{div}}\nolimits}

\begin{document}
 \title[Chow stability]{Extensions of two Chow stability criteria to positive characteristics}
 \author{Shinnosuke OKAWA}
 \address{Graduate School of Mathematical Sciences, 
the University of Tokyo, 3-8-1 Komaba, Meguro-Ku, Tokyo 153-8914, Japan.}
\email{okawa@ms.u-tokyo.ac.jp}
\date{\today}
\subjclass[2000]{Primary 14L24; Secondary 13A50, 14B05}
\keywords{Chow stability, log canonical threshold}

 \maketitle
 
  %%%%%%%%%%%%%%%%%%%%%%%%%%%%abstract

\begin{abstract}
We extend two results on Chow (semi-)stability to positive
characteristics. One is on the stability of non-singular projective
hypersurfaces of degree at least three, and the other is the criterion
by Y. Lee in terms of log canonical thresholds.
Some properties of log-canonicity in positive characteristics are discussed
with a couple of examples, in connection with the proof of the latter one.
It is also proven in Appendix that the sum of Chow (semi-)stable cycles are again
Chow (semi-)stable.
\end{abstract}

%%%%%%%%%%%%%%%%%%Introduction

\section{Introduction}
We work over an algebraically closed field $k$ of an arbitrary characteristic.

Let $X\subset\mathbb{P}^{n}_{k}$ be an effective cycle of dimension $r$ and degree $d$
in a projective space of dimension $n$.
Analysis of the Chow (semi-)stability of X is one of the basic problems in Geometric Invariant
Theory (GIT). Contrary to the asymptotic Chow (semi-)stability, the precise classification of
Chow (semi-)stable cycles is quite a subtle problem, and is known
only for few cases, even for projective hypersurfaces. To name a few,
J. Shah studied the case of plane sextics (\cite{shah}) and recently
R. Laza did the case of cubic fourfolds (\cite{rl}), both in relation with
period maps.

On the other hand, there are two sufficient conditions for Chow (semi-)stability
in terms of the singularity of $X$ or that of Chow divisor $Z(X)\subset\mathbb{G}
=\mathrm{Grass}_{k}(n-r,n+1)$,
which deal with general situations. Both have been proven in characteristic zero,
and the purpose of this paper is to extend them to arbitrary characteristics. Namely
we prove the following two theorems:
%hp
\begin{theo}[$=$ Theorem \ref{hp}]\label{hp in Intro}
If $d\ge 3$, any non-singular projective hypersurface of degree $d$ is 
Chow stable.
\end{theo}
%yl
\begin{theo}[$=$ Theorem \ref{yl}]\label{yl in Intro}
Let $X$ be an effective cycle of dimension $r$ and degree $d$ in 
$\mathbb{P}^n_{k}$. Let $(\mathbb{G},Z(X))$ be the log pair defined by the
Chow divisor $Z(X)$ of $X$. If $\lct(\mathbb{G},Z(X))>\frac{n+1}{d}$
(resp. $\ge\frac{n+1}{d}$), then $X$ is Chow stable (resp. Chow semi-stable).
\end{theo}
In the statement of Theorem \ref{yl in Intro}, $\lct(\mathbb{G},Z(X))$ stands for
the log canonical threshold of $(\mathbb{G},Z(X))$, which measures how good
the singularity of $Z(X)$ is (see \S\ref{Notions of singularities} for
detail).

Characteristic zero case of Theorem \ref{hp in Intro} is due to Mumford (\cite[Chapter 4 \S 2]{git}),
and that of Theorem \ref{yl in Intro} is due to Y. Lee (\cite{yl}).

The original proof of Theorem \ref{hp in Intro} works only when the characteristic of the base field
does not divide $d$ (see \S \ref{Chow stability of non-singular hypersurfaces}). To prove the general case, we depend on the
corresponding result in characteristic zero.

We sketch the proof of Theorem \ref{hp in Intro} in positive characteristics.
First we take a suitable lift of the equation of given hypersurface over the ring of
Witt vectors. This defines a family of projective hypersurfaces over the ring.
We are assuming that the closed fiber is non-singular, hence the geometric generic fiber
is again non-singular. Since we know that Theorem \ref{hp in Intro} holds in characteristic
zero, we obtain
some inequalities for the Hilbert-Mumford numerical functions of the lift. By the
choice of the lift, those numerical functions coincide with those of the original hypersurface.
Thus we obtain the inequalities for the numerical functions of the original one, concluding
the proof.

The point is that the singularity of the hypersurface over the generic point is better than that
of the special fiber, so that we can use the corresponding stability criterion
in characteristic zero.
This method seems to be applicable to other stability problems
(see the remark at the beginning of \S
\ref{Y. Lee's criterion in characteristic $p$}).

In \S\ref{The defining equation} it will also be shown that the complement of the locus of non-singular hypersurfaces
is an irreducible divisor, even when $p$ divides $d$. In general some multiple of the defining equation of
this divisor lifts to the usual discriminant in characteristic zero. 

Theorem \ref{yl in Intro} will be proven along the same line as the proof given in \cite{yl}, but
we must modify several points. This is due to the fact that some properties of log canonicity
which hold in characteristic zero fail in positive characteristics, because of
the existence of wild ramifications and inseparable morphisms.

We can prove that the property of log canonicity which we need still holds for
finite separable morphisms. It turns out that this is enough for our purpose,
for we can use a perturbation technique so that we need not to deal with
the inseparable morphisms (see \S \ref{Y. Lee's criterion in characteristic $p$}).

In \S \ref{Y. Lee's criterion in characteristic $p$}
we also discuss some other properties of log canonicity, with a couple of (counter-)examples.

In Appendix \ref{sum} we prove the following
\begin{prop}[$=$ Proposition \ref{stability of sum}]\label{stability of sum in Intro}
Let $Y,Z$ be Chow semi-stable cycles of the same dimension in a projective
space $\mathbb{P}^{n}_{k}$. Then $Y+Z$ is again Chow semi-stable.
Furthermore if $Y$ is Chow stable, so is $Y+Z$.
\end{prop}

This proposition may be well-known to experts, but the author could not find it in the literature. 
The proof is a simple application of the fact that the stability can be checked
1-PS wise, which is essentially the same as the numerical criterion.
But the conclusion itself seems to be rather
surprising: if we have two Chow stable cycles, the sum of them is always Chow stable no matter
how badly they touch.

Proposition \ref{stability of sum in Intro} will be used to give a family of stable projective hypersurfaces
whose stability can not be detected by Theorem \ref{yl in Intro} (see Example \ref{counter example}).

%%%%%%%%%%%%%%%%%%%%%%%Acknowledgement

\subsection*{Acknowledgement}
The author would like to express his sincere gratitude
to his supervisor Professor Yujiro Kawamata for his
valuable advice, many suggestions to improve this paper and warm encouragement.
He is indebted to Prof. Yongnam Lee for kindly answering the author's question and Prof. Shunsuke
Takagi for informing the author of the paper \cite{mus} in relation with Conjecture \ref{lsc}.
He would also like to thank the referee for useful comments and corrections for the draft of this paper.
Thanks are also due to Prof. Kei-ichi Watanabe for his useful comments and
interest in my work, and to Dr. Shohei Ma for his nice comments. 

This paper is an extended version of the author's master thesis submitted to the University of Tokyo
in 2010.

%%%%%%%%%%%%%%%%%%Preliminary

\section{Preliminary}

\subsection{Notations from scheme theory}
We need some notations from \cite{ha}.

Let $R$ be an $\mathbb{N}$-graded ring. For a homogeneous ideal $I$ of $R$ we denote
by $V(I)$ the corresponding closed subscheme $\mathrm{Proj}(R/I)$ of $\mathrm{Proj}R$.

For a homogeneous element $f\in R$,
we denote $\mathrm{Proj}R\setminus V(f)$ by $D_+(f)$.
This open subscheme is known to be affine, with coordinate ring
\begin{equation*}
R_{(f)}=\left\{\frac{r}{f^n} | r\in R , \ \ \deg (r)=n\cdot\deg (f)\right\}.
\end{equation*}

%%%%%%%%%%%%%%%%%%%%%Notions of singularities

\subsection{Notions of singularities}\label{Notions of singularities}

In this subsection, we summarize the notions of singularities of pairs which we need later.

\begin{defi}[discrepancy, log canonical]

Let $X$ be a normal variety over $k$ and $\Delta$ be an effective
$\mathbb{R}$-Weil divisor on $X$ such that $K_X+\Delta$ is $\mathbb{R}$-Cartier.

Let $\pi:Y\to X$ be a birational morphism from another normal variety $Y$
over $k$ and
$E\subset Y$ be a prime divisor. Then in a neighborhood of the generic point
of $E$, the following canonical bundle formula holds:
\begin{equation*}
K_Y=\pi^{*}(K_X+\Delta)+aE.
\end{equation*}
The real number $a$ in the above equation is called the \textit{discrepancy}
of $E$ with respect to $(X,\Delta)$, and
denoted by $a(E;X,\Delta)$. It is independent of the choice of $Y$ and $\pi$, depending only on the valuation
of $k(X)$ which corresponds to $E$.

We say that the log pair $(X,\Delta)$ is \textit{log canonical} (\textit{lc}, for short) if
$a(E;X,\Delta)\ge -1$ holds for all the divisors $E$ as above.
\end{defi}

A finer version is:
\begin{defi}
Let $x\in X$ be a point. We say that the log pair $(X,\Delta)$ is \textit{log canonical at $x$} if
the restriction of $(X,\Delta)$ to an open neighborhood of $x$ is log canonical.
\end{defi}

\begin{defi}[log canonical threshold]
Let $(X,\Delta)$ be a log canonical pair and $D$ be an effective $\mathbb{R}$-Cartier divisor on $X$. 
The \textit{log canonical threshold} of $D$ with respect to $(X,\Delta)$ is defined as follows:

$\lct(X,\Delta;D)=\sup\{t\in\mathbb{R}|(X,\Delta+tD)$ is log canonical$\}$.

For a point $x\in X$, we set

$\lct_{x}(X,\Delta;D)=\sup\{t\in\mathbb{R}|(X,\Delta+tD)$ is log canonical at $x$$\}$.
\end{defi}

It is easy to see by definition that "$\sup$" in the above definition is actually "$\max$".

When we consider the case $\Delta=0$, we write $\lct(X,\Delta;D)=\lct(X,D)$ for short (resp.
$\lct_{x}(X,\Delta;D)=\lct_{x}(X,D)$).

%%%%%%%%%%%%%%%%%%Chow stability and numerical criterion

\subsection{Chow stability and the numerical criterion}\label{Chow stability and the numerical criterion}
Let $X\subset\mathbb{P}^n_{k}$ be an effective
$r$-dimensional cycle of degree $d$. 
We associate to $X$ its \textit{Chow divisor} $Z(X)$,
which is a hypersurface of degree $d$ of 
the Grassmannian $\mathbb{G}=\mathrm{Grass}_k(n-r,n+1)$, as follows
(one may consult either \cite{ko2} or \cite{gkz} for detail).
If $X$ itself is a variety, set $Z(X)=\{L\in\mathbb{G}|L\cap X\not=\phi\}$.
For a general cycle $X$, define $Z(X)$ additively. 
The defining equation of $Z(X)$ is called the \textit{Chow form}
of $X$ (Chow form is determined by $X$ only up to scalar multiplication).
The homogeneous coordinate ring of $\mathbb{G}$ with respect to the
Pl{\"u}cker embedding is denoted by $\mathcal{B}=\sum_{d\ge 0}\mathcal{B}_d$. This
is the subring of the polynomial ring of $(n+1)(n-r)$ indeterminants $U_i^{(j)}$'s, 
where $(i,j)$ runs through the range $i=0,\dots, n$ and $j=1,\dots, n-r$, generated
by all the $(n-r)\times(n-r)$ minors of the matrix $(U_i^{(j)})$.
The Chow form of a cylce $X$ is an element of $\mathcal{B}_d$ (up to scalar multiplication), so that
the Chow divisor $Z(X)$ of $X$ can be regarded as an element of the projective space
$\mathbb{P}_{*}\mathcal{B}_d$.
The canonical action of $SL(n+1,k)$ on $\mathbb{P}^n_{k}$ naturally induces a linear 
action on $\mathcal{B}_d$, hence we can discuss the GIT (semi-)stability of an element
of $\mathbb{P}_{*}\mathcal{B}_d$ (here we are using the terminology ``stable'' in the sense 
of ``properly stable'' in \cite{git}, which requires the finiteness of the stabilizer subgroup.
We heavily rely on the numerical criterion, so we follow this terminology\footnote{The
author would like to thank Dr. S. Ma for this remark.}).
 Chow (semi-)stability of $X$ is defined to be the 
(semi-)stability of $Z(X)$ in the above sense.

Next we recall the \textit{Hilbert-Mumford numerical criterion} (\textit{numerical criterion},
for short) for
stability and obtain an explicit description of the numerical function $\mu$ following
\cite[Proposition 2.3]{git}. We start with some preparations.

For a non-negative integer $n$, set
\begin{equation}\label{def of [n]}
[n]=\{0,1,\dots,n\}.
\end{equation}
For a subset $I\subset [n]$ with $\#I=n-r$, let $\Delta_I$ be the $(n-r)\times(n-r)$ minor
of the matrix $(U_i^{(j)})$  obtained by picking out the $n-r$ rows according to $I$.
Recall that $\mathcal{B}_d$ is a $k$-vector space generated by the set
$\{\Delta_{I_1}\dots\Delta_{I_d}|I_{\ell}\subset [n], \#I_{\ell}=n-r$ for all $\ell=1,\dots,
d\}$.

Now fix $X$. Take any $g\in SL(n+1,k)$ and let $F$ be the Chow form of $g^{*}X$
($=$ the defining equation of $g^{*}Z(X)$).
Set 
\begin{equation}\label{R}
\mathcal{R}=\{\vec{r}=(r_0,\dots,r_n)\in\mathbb{Z}^{n+1}\setminus \{0\}|
\sum_{i=0}^{n}r_i=0, \ \ r_0\le r_1\le\dots\le r_n\}.
\end{equation}
 An element $\vec{r}$ of $\mathcal{R}$ corresponds to a non-trivial one-parameter subgroup
(1-PS for short) $\lambda:\mathbb{G}_m\to SL(n+1,k)$ of $SL(n+1,k)$
which is defined by $\lambda(t)=\diag(t^{r_0},
\dots,t^{r_n})$. If we regard $\mathcal{B}_d$ as a representation of $\mathbb{G}_m$
via $\lambda$,
the one dimensional subspace of $\mathcal{B}_d$ spanned by $\Delta_{I_1}\cdots\Delta_{I_d}$
is an eigenspace of the weight 
\begin{equation*}
\wt(I_1,\dots,I_d)
=\sum_{\ell=1}^{d}\sum_{i\in I_{\ell}}r_i
=\sum_{i=0}^{n}r_i\cdot\#\{\ell|i\in I_{\ell}\}.
\end{equation*}

Using these notations, the numerical function of $X$ with respect to $g\in SL(n+1,k)$ and $\vec{r}\in\mathcal{R}$
is defined as follows:
\begin{defi}[numerical function]\label{numerical function}
Let $\mathcal{I}(F)$ be the set of such $d$-tuples $(I_{1},\dots,I_{d})$ that
the coefficient of $\Delta_{I_1}\cdots\Delta_{I_d}$ in $F$ is not zero. Then set
\begin{equation*}
\mu(Z(X),g,\vec{r})=\mu(V(F),\id,\vec{r})=\min_{(I_{1},\dots,I_{d})\in \mathcal{I}(F)}
\wt(I_1,\dots,I_d).
\end{equation*} 
\end{defi}

\begin{rem}\label{key remark}
$\mu(Z(X),g,\vec{r})$ depends only on $g,\vec{r}$ and the set $\mathcal{I}(F)$.
\end{rem}

Now the numerical criterion is:
\begin{prop}\label{nc}
$X$ is Chow stable (resp. semi-stable) if and only if 
$\mu(Z(X),g,\vec{r})<0$ (resp. $\le 0$) holds for any $g\in SL(n+1,k)$ and $\vec{r}\in\mathcal{R}$.
\end{prop}

%%%%%%%%%%%%%%%%%%%%%%%%%%%%%nc by yl

Next we rephrase Proposition \ref{nc} in such a way as to prove Theorem \ref{yl}.
This reinterpretation is just a generalization of \cite[Lemma 2.1]{yl}.
Before that, we need some preparations.
Take an arbitrary $g\in SL(n+1,k)$ and let $F$ be the Chow form of $g^{*}X$.

Let $f$ be the local equation of $F$ on $D_{+}(\Delta_{[n-r-1]})\simeq \spec{
\mathcal{B}_{(\Delta_{[n-r-1]})}}$.
Recall that $\mathcal{B}_{(\Delta_{[n-r-1]})}$ is the polynomial ring over $k$ with the set of
indeterminants $\left\{x_{I}=\frac{\Delta_{I}}{\Delta_{[n-r-1]}}|I\right\}$,
where $I$ runs through those subsets of $[n]$ (see (\ref{def of [n]})) satisfying the following two conditions:
\begin{equation}\label{conditions on I}
\begin{split}
\#I&= n-r\\
\#(I\cap[n-r-1])&= n-r-1.\\
\end{split}
\end{equation}
Therefore $f$ is a polynomial in $x_I$'s.
Now assign nontrivial integral weights $\vec{r}=(r_0,\dots,r_n)\in\mathcal{R}$
to $X_0,\dots,X_n$, so that the induced weight $w(x_I)$ on $x_I$ satisfies
\begin{equation}\label{weight of x_I}
w(x_I)=\sum_{i\in I}r_i-\sum_{i=0}^{n-r-1}r_i,
\end{equation}
which is non-negative by the assumption $r_0\le r_1\le\dots\le r_n$.

%%%%%%%%%%%%%%%nc by yl

Now Proposition \ref{nc} is equivalent to
\begin{lem}\label{nc by yl}
A cycle $X$ is Chow stable (resp. semi-stable) if and only if
\begin{equation}\label{eq of nc by yl}
\frac{w(f)}{\sum_{I}w(x_I)}<\frac{d}{n+1}
\end{equation}
(resp. $\le\frac{d}{n+1}$) holds for all $g\in SL(n+1,k)$ and $\vec{r}\in\mathcal{R}$ (see (\ref{R})
for the definition of $\mathcal{R}$).
\end{lem}

Above $f$ is the local equation on $D_{+}(\Delta_{[n-r-1]})$ of the Chow form of $g^{*}X$ as before.
In the left hand side of (\ref{eq of nc by yl}),
$w(f)$ denotes the weighted multiplicity of $f$ (= the lowest weight of the monomials occurring in $f$)
with respect to the weight $\left(w(x_I)\right)_I$.

\begin{proof}

We only discuss the stable case. Semi-stable case can be proven similarly.

The inequality (\ref{eq of nc by yl}) is equivalent to
\begin{equation}\label{eq1 of pf of nc by yl}
d\sum_{I}w(x_I)-(n+1)w(f)>0.
\end{equation}
Combining the calculation of $w(x_I)$ (see (\ref{weight of x_I})) with
 the definition of $w(f)$, we see that the left hand side of (\ref{eq1 of pf of nc by yl})
 equals to
 
\begin{eqnarray*}
 d\left(\sum_{I}\sum_{i\in I}r_i-(n-r)(r+1)\sum_{i=0}^{n-r-1}r_i\right) 
 -(n+1)\left(\mu(X,g,\vec{r})-d\sum_{i=0}^{n-r-1}r_i\right).
\end{eqnarray*}

Recalling the conditions (\ref{conditions on I}) posed on $I$'s we see
\begin{equation*}
\sum_{I}\sum_{i\in I}r_i =
(n-r-1)(r+1)\sum_{i=0}^{n-r-1}r_i+(n-r)\sum_{i=n-r}^{n+1}r_i.
\end{equation*}

A little calculation shows that the left hand side of (\ref{eq1 of pf of nc by yl})
 boils down to

\begin{equation*}
d(n-r)\sum_{i=0}^{n}r_i-(n+1)\mu(X,g,\vec{r})
=-(n+1)\mu(X,g,\vec{r}),
\end{equation*}
since we assumed that $\sum_{i=0}^{n}r_i=0$.

Therefore
(\ref{eq of nc by yl}) is equivalent to the condition $\mu(X,g,\vec{r})<0$.

\end{proof}

%%%%%%%%%%%%%%%%%%Chow stability in characteristic p from characteristic zero

\subsection{Chow stability in characteristic $p$ from characteristic zero}

Let $k$ be a field of characteristic $p>0$ and $X$ be a cycle in $\mathbb{P}^{n}_{k}$.
In this subsection we want to propose a method to deduce the
Chow (semi-)stability of $X$ from the corresponding results in characteristic zero.
 
From now on, we denote by $W=W(k)$ the ring of Witt vectors.
This is a discrete valuation ring (DVR for short) of characteristic zero, whose residue field
is isomorphic to $k$ (see \cite[Chapter 2 \S 5 Theorem 5]{s}). Actually these are all the
properties of $W$ which we need in this paper. We denote by $K$ the field of fractions of
$W$ and by $m_W$ the unique maximal ideal of $W$.

Take $g\in SL(n+1,k)$ and let $F$ be the Chow form of $g^{*}X$ as in the previous
subsection. Let $F_W$ be a lift of $F$ over $W$ such that
a monomial which does not appear in $F$ never appears in $F_W$,
which is equivalent to the assumption $\mathcal{I}(F)=\mathcal{I}(F_W)$
(see Definition \ref{numerical function} for the definition of $\mathcal{I}$).
 Note that $F_W$ defines a hypersurface $V(F_W)\subset
\mathrm{Grass}_{\overline{K}}(n-r,n+1)$ of degree $d$,
where $\overline{K}$ is the algebraic closure of $K$. 

\begin{theo}\label{red.technique}
Assume that for any $g\in SL(n+1,k)$ we can take $F_W$ such that 
$\mathcal{I}(F)=\mathcal{I}(F_W)$ holds and $V(F_W)$
is stable (resp. semi-stable) with respect to the induced action of $SL(n+1,\overline{K})$.
Then $X$ is Chow stable (resp. Chow semi-stable).
\end{theo}
\begin{proof}
Since $F_W$ is (semi-)stable, $\mu(V(F_W),\id,\vec{r})<0$ (resp. $\le 0$)
holds for any $\vec{r}\in\mathcal{R}$ (see  (\ref{R}) in the previous subsection
for the definition of $\mathcal{R}$).
But it holds that $\mu(V(F_W),\id,\vec{r})=\mu(Z(X),g,\vec{r})$,
since $\mathcal{I}(F)=\mathcal{I}(F_W)$ (see Remark \ref{key remark}). Therefore
$\mu(Z(X),g,\vec{r})<0$ (resp. $\le 0$) holds for all $g\in SL(n+1,k)$ and $\vec{r}\in\mathcal{R}$, hence
we see the Chow (semi-)stability of $X$ by Proposition \ref{nc}.
\end{proof}
\begin{rem}
By the result of C. S. Seshadri (\cite[Proposition 6]{sesh}, see also \cite[Appendix to Chapter 1, \S G]{git}),
the converse of Theorem \ref{red.technique} also holds: if $X$ is Chow stable (resp. Chow
semi-stable), any lift $F_W$ of $F$ is also stable (resp. semi-stable) with respect to the
induced action of $SL(n+1,\overline{K})$.

\end{rem}

%%%%%%%%%%%%%%%%%%%%%%%%%%%%%%%%%%%%%%%%
%%%%%%%%%%%%%%%%%%%Chow stability of nonsingular hypersurfaces

\section{Chow stability of non-singular hypersurfaces}\label{Chow stability of non-singular hypersurfaces}
In this section $X$ denotes a hypersurface of degree $d$ in $\mathbb{P}^{n}_{k}$.

In \S\ref{A proof via lifting to characteristic zero}, we prove the stability of
non-singular hypersurfaces of degree at least three. This is an easy application of Theorem \ref{red.technique}.
In \S\ref{The defining equation} we study the complement of the locus of non-singular hypersurfaces
via geometric arguments. It turns out that the complement is an irreducible divisor and that some multiple of its
defining equation lifts to the usual discriminant in characteristic zero.

\subsection{A proof via lifting to characteristic zero}\label{A proof via lifting to characteristic zero}

First of all we recall that the characteristic zero case of Theorem \ref{hp} was settled in \cite[Chapter 4 \S 2]{git}.
Thanks to a theorem by Matsumura and Monsky, the proof given there also works for 
characteristic $p$ cases if $p$ does not divide $d$. We briefly recall the proof
and see why it does not work for the cases when $p$ do divide $d$.

Let $F(X_0,X_1,\dots,X_n)$ be a homogeneous polynomial of degree $d$. We have the
Euler's lemma:
\begin{equation*}
dF=\sum_{i=0}^{n}X_i\frac{\partial F}{\partial X_i}.
\end{equation*}

Therefore we see that 
\begin{equation}\label{partial der}
V\left(F,\frac{\partial F}{\partial X_0},\frac{\partial F}{\partial X_1},
\dots,\frac{\partial F}{\partial X_n}\right)=
V\left(\frac{\partial F}{\partial X_0},\frac{\partial F}{\partial X_1},
\dots,\frac{\partial F}{\partial X_n}\right),
\end{equation}
 provided that $p$ does not divide $d$.
The emptiness of the latter is equivalent to the vanishing of the discriminant of
$F$ when $d\ge 2$. This shows the semi-stability of 
non-singular hypersurfaces of degree at least $2$. Furthermore, 
when $d\ge 3$,
it is known (see \cite[Theorem 1]{mm}) that only finitely many projective linear transformations 
preserve the given non-singular hypersurface. This means that any 
non-singular hypersurface is stable, provided $d\ge 3$ and $p\!\not| d$.

The above argument does not work in general, for the equality (\ref{partial der}) may
break down when $p$ divides $d$. Actually when $p$ divides $d$ and $d\ge 3$, the right hand side of
the equality (\ref{partial der}) can not be empty. This will be proven in the next subsection
(see Proposition \ref{non-emptiness}).

Even when $p$ divides $d$, a closer look at the numerical criterion
shows that non-singular hypersurfaces are always (semi-)stable if $d>n+1$ (resp. $d\ge n+1$)
(see \cite[Lemma 4.2]{n}. This may also be deduced from Theorem \ref{yl},
since the pair $(\mathbb{P}^n_{k},X)$ is log canonical when $X$ is a non-singular hypersurface).

Now we prove that the stability is always the case:
\begin{theo}\label{hp}
If $d\ge 3$, any non-singular projective hypersurface of degree $d$ is 
Chow stable.
\end{theo}
\begin{proof}
The theorem is already established when $\ch k=0$, so we assume $\ch k>0$.
We use Theorem \ref{red.technique}.
Let $X\subset\mathbb{P}^{n}_k$ be an non-singular projective hypersurface of 
degree at least three. 
Take any $g\in SL(n+1,k)$ and let $F_k$ be the equation of $g^{*}X$.
Note that in this case $F_k$ itself is the Chow form of $g^{*}X$.
Take a lift $F_W$ of $F_k$ over the ring of Witt vectors $W$ satisfying
$\mathcal{I}(F_k)=\mathcal{I}(F_W)$
(see Definition \ref{numerical function} for the definition of $\mathcal{I}$).
Then it is easy to see the
\begin{claim}
$V(F_W)$ is an integral scheme.
\end{claim}
\begin{proof}
Since $W$ is a DVR, $W[X_0,\dots,X_n]$ is a UFD.
So it is enough to show that $F_W$ is an irreducible element of $W[X_0,\dots,X_n]$.
Suppose for a contradiction that $F_W=G\cdot H$ holds for some $G,H\in W[X_0,\dots,X_n]$
such that neither $G$ nor $H$ is a unit. Note that both $G$ and $H$ are homogeneous and non-zero, since
$F_W$ is.  Therefore $G, H$ are homogeneous polynomials of degrees at least one, since neither of them
is a unit. This means that either $\overline{G}=0$ or $\deg{\overline{G}}=\deg{G}$ (here $\overline{G}$ denotes the
reduction modulo $m_W$ of $G$) must hold (similar
for $H$).
On the other hand, $\overline{G}\cdot\overline{H}=\overline{F_W}=F_k\not=0$
holds. Therefore $\deg{\overline{G}}=\deg{G}\ge 1$ (resp. $\deg{\overline{H}}\ge 1$),
contradicting the irreducibility of $F_k$.
\end{proof}

Since $V(F_W)$ dominates the generic point of $\spec{W}$, the above claim means that
$V(F_W)$ is flat over $\spec W$ (see \cite[Chapter III, Proposition 9.7]{ha}). Also it is projective over $\spec{W}$.

The closed fiber of $V(F_W)\to \spec W$ is $g^{*}X$,
which is non-singular. Therefore the geometric generic fiber is also non-singular
(see \cite[(12.2.4)(iii)]{ega}). Since the characteristic of the generic fiber is zero and $\deg(F_W)\ge 3$,
we already know that it is stable. By Theorem \ref{red.technique}, we see that $X$ is stable too.
\end{proof}

%%%%%%%%%%%%%%%%%%%%%%%%\subsection{The defining equation}

\subsection{The defining equation}\label{The defining equation}
Let $\mathrm{Hyp}_{d}(n)$ be the projective space of degree $d$ hypersurfaces in $\mathbb{P}^{n}_{k}$,
and $\mathrm{U}_{ns}\subset \mathrm{Hyp}_{d}(n)$ be the locus of non-singular hypersurfaces. 
In this subsection we study the defining equation for the complement of the locus of
non-singular hypersurfaces, $\mathrm{Hyp}_{d}(n)\setminus \mathrm{U}_{ns}$, via geometric arguments.
This is a version of the arguments given in \cite[Chapter 5 \S 2]{mukai}.

The defining equation is well-known when
$p$ does not divide $d$, the discriminant. Therefore we are interested in the cases when $p$ divide $d$. 

Recall that the non-singularity of $X=V(F)$ is
equivalent to the emptiness of the left hand side of (\ref{partial der}).
Using this, we show the following

\begin{theo}\label{complement}
Assume $p$ divides $d$. Then
\begin{equation*}
\mathrm{Hyp}_{d}(n)\setminus \mathrm{U}_{ns}
\end{equation*}
is an irreducible divisor. Moreover some multiple of its defining equation lifts to the
discriminant in characteristic zero.
\end{theo}
\begin{exmp}(See \cite[Chapter 10 \S 2]{dol} for detail.)
Consider the case $(n,d)=(1,4)$. Let
\begin{equation*}
X=V(F), F=a_0X_0^4+a_1X_0^3X_1+a_2X_0^2X_1^2+a_3X_0X_1^3+a_4X_1^4
\end{equation*}
be an hypersurface in $\mathbb{P}^{1}_{k}$.
When $\ch{k}\not=2$, the defining equation for $Hyp_{4}(1)\setminus \mathrm{U}_{ns}$ is
given by $D=4S^3-T^2$, where
\begin{eqnarray*}
S&=&2^2\cdot 3a_0a_4-3a_1a_3+a_2^2\\
T&=&2^3\cdot 3^2a_0a_2a_4-3^3a_0a_3^2+3^2a_1a_2a_3-3^3a_1^2a_4-2a_2^3.\\
\end{eqnarray*}
When $\ch{k}=2$, $D\mod 2=(T\mod 2)^2$ and the defining equation for
$Hyp_{4}(1)\setminus U$ is given by
$T\mod 2=a_0a_3^2+a_1a_2a_3+a_1^2a_4$.
\end{exmp}

\begin{proof}[Proof of Theorem \ref{complement}]
Let $W=W(k)$ be the ring of Witt vectors. Set
\begin{equation*}
I=\left\{(x,X);x\in V\left(F,\frac{\partial F}{\partial X_0},\frac{\partial F}{\partial X_1},
\dots,\frac{\partial F}{\partial X_n}\right)\right\}\subset\mathbb{P}^{n}_{W}\times_{\spec{W}}
\mathrm{Hyp}_{d}(n),
\end{equation*}
where $\mathrm{Hyp}_{d}(n)=|\mathcal{O}_{\mathbb{P}^{n}_{W}}(d)|$ is the projective space of
families of degree $d$ projective hypersurfaces over $\spec W$. Let $p:I\to\mathbb{P}^{n}_{W},
q:I\to \mathrm{Hyp}_{d}(n)$ be the natural projections. 

First we show the following
\begin{claim}
$p$ is a smooth morphism with connected fibers.
\end{claim}

\begin{proof}
Let $x:\spec{\Omega}\to\mathbb{P}^{n}_{W}$ be a geometric point, where $\Omega$ is an algebraically closed
field. By the definition of
$I$ above, it is easy to see that $I_x\subset \mathrm{Hyp}_{d}(n)_{\Omega}:=\mathrm{Hyp}_{d}(n)\times_{\spec{W}}{\spec\Omega}$ is a
linear subspace. 

Next we calculate the fiber $I_{(1:0:\cdots:0)}$, where $(1:0:\cdots:0)\in
\mathbb{P}^{n}_{\Omega}$.
Note that if we write
\begin{equation*}
F(X)=\sum_{|\alpha|=d}C_{\alpha}X^{\alpha}
\end{equation*}
by using multi-indices, $(C_{\alpha};|\alpha|=d)$ gives a system of coordinates for the
projective space $\mathrm{Hyp}_{d}(n)_{\Omega}$. 
Then we can show the following equality (note that the equality is independent of the characteristic of $\Omega$):

\begin{equation*}
I_{(1:0:\cdots:0)}=V(C_{(d0\cdots 0)}, C_{((d-1)10\cdots 0)}, C_{((d-1)010\cdots 0)}, \dots, C_{((d-1)0\cdots 01)})\subset
\mathrm{Hyp}_{d}(n)_{\Omega}.
\end{equation*}

 In order to show that the dimension of the linear subspace $I_x$ is independent of $x$,
we show that it is isomorphic to $I_{(1:0:\cdots:0)}$.
Consider the action of $SL_{\Omega}(n+1)$ on $\mathbb{P}^{n}_{W}\times_{\spec{W}} \mathrm{Hyp}_{d}(n)_{\Omega}$
which is defined by $g\cdot (x,X)=(gx,g_{*}X)$ for $g\in SL(n+1,\Omega)$.
It can be easily checked that this action preserves $I\times_{\spec{W}}{\spec\Omega}$ and that we obtain an isomorphism between
$I_{(1:0:\cdots:0)}$ and $I_{x}$ via this action.
\end{proof}

By the claim we see that both $I$ and $I_{k}$, the restriction of $I$ over the closed point $\spec{k}\subset\spec{W}$,
are integral schemes.

Now consider the integral closed subscheme $q(I)\subset \mathrm{Hyp}_{d}(n)$. Note that
the defining equation for $q(I)$ is the usual discriminant, and that $q(I)_{k}$, restriction of
$q(I)$ over $\spec{k}\subset\spec{W}$, coincides with $q(I_{k})$ as sets.
\end{proof}

Similar arguments as above show the following

\begin{prop}\label{non-emptiness}
Assume $d$ is divided by $p$ and either $d\ge 3$ or $d=p=2$ and $n$ is even.
Let $X=V(F)\subset\mathbb{P}^{n}_{k}$ be an
arbitrary hypersurface of degree $d$. Then 
\begin{equation*}
V\left(\frac{\partial F}{\partial X_0},\frac{\partial F}{\partial X_1},
\dots,\frac{\partial F}{\partial X_n}\right)\not=\emptyset
\end{equation*}
holds.
\end{prop}

\begin{proof}
Set 
\begin{equation*}
Z=\left\{(x,X);x\in V\left(\frac{\partial F}{\partial X_0},\frac{\partial F}{\partial X_1},
\dots,\frac{\partial F}{\partial X_n}\right)\right\}\subset\mathbb{P}^{n}_{k}\times \mathrm{Hyp}_{d}(n)
\end{equation*}
and let $p:Z\to\mathbb{P}^{n}_{k}$ and $q:Z\to \mathrm{Hyp}_{d}(n)$ be the natural projections.
As in the proof of Theorem \ref{complement},
we can show the following
\begin{claim}
$p$ is a smooth morphism with connected fibers.
\end{claim}

Next we calculate the dimension of $Z_{(1:0:\cdots:0)}$. With the notations in the proof of Theorem \ref{complement},
we can write
\begin{equation*}
Z_{(1:0:\cdots:0)}=V(C_{((d-1)10\cdots 0)},C_{((d-1)010\cdots 0)},\dots,C_{((d-1)0\cdots 01)})\subset \mathrm{Hyp}_{d}(n).
\end{equation*}
Therefore 
\begin{eqnarray*}
\dim Z &=& \dim\mathbb{P}^{n}_{k}+\dim{Z_{(1:0:\cdots:0)}}\\
&=& n+(\dim \mathrm{Hyp}_{d}(n)-n)\\
&=& \dim \mathrm{Hyp}_{d}(n).\\
\end{eqnarray*}

Now all we have to show is that $q:Z\to q(Z)$ is generically finite, because then
we see that $\dim q(Z)=\dim Z=\dim \mathrm{Hyp}_{d}(n)$, hence $q(Z)=\mathrm{Hyp}_{d}(n)$.
In order to show it, we check the finiteness of the fiber of $q$ at
\begin{equation*}
F(X)=X_0^{d-1}X_1+X_1^{d-1}X_2+\cdots+X_{n-1}^{d-1}X_n+X_n^{d-1}X_0.
\end{equation*}

First of all, note that 
\begin{equation}\label{recursion}
\frac{\partial F}{\partial X_i}=0\iff
X_{i-1}^{d-1}=X_{i}^{d-2}X_{i+1}
\end{equation}
for all $i=0,1,\dots,n$, where $X_{-1}=X_n$ and $X_{n+1}=X_0$.
Suppose $a=(a_0:\cdots:a_n)\in V\left(\frac{\partial F}{\partial X_0},\frac{\partial F}{\partial X_1},
\dots,\frac{\partial F}{\partial X_n}\right)$.

From (\ref{recursion}), one see that $a_0\cdots a_n\not= 0$. Hence we assume $a_0=1$.
Using (\ref{recursion}) recursively, we obtain the following equation:
\begin{equation}\label{root of unity}
a_1^{1-(1-d)^{n+1}}=1.
\end{equation}
The exponent of $a_1$ above is non-zero under our assumptions on $(d,p,n)$. Therefore 
(\ref{root of unity}) poses a non-trivial condition on $a_1$. $a_2, a_3, \dots, a_{n}$ are
uniquely determined from $a_1$, because of $(\ref{recursion})$. Thus the finiteness is shown.
\end{proof}

\begin{rem}
When $d=p=2$ and $n$ is odd, we have the following counter-example:
\begin{equation*}
F=X_0X_1+X_2X_3+\cdots+X_{n-1}X_{n}.
\end{equation*}
It is easy to see that
$V\left(\frac{\partial F}{\partial X_0},\frac{\partial F}{\partial X_1},
\dots,\frac{\partial F}{\partial X_n}\right)=\emptyset$.

\end{rem}
%%%%%%%%%%%%%%%%%Y. Lee's criterion in characteristic p

\section{Y. Lee's criterion in characteristic $p$}\label{Y. Lee's criterion in characteristic $p$}
In this section we prove the following
\begin{theo}\label{yl}
Let $X$ be an effective cycle of dimension $r$ and degree $d$ in 
$\mathbb{P}^n_{k}$. Let $(\mathbb{G},Z(X))$ be the log pair defined by the
Chow divisor $Z(X)$ of $X$. If $\lct(\mathbb{G},Z(X))>\frac{n+1}{d}$
(resp. $\ge\frac{n+1}{d}$), then $X$ is Chow stable (resp. Chow semi-stable).
\end{theo}

See \S\ref{Chow stability and the numerical criterion} for notations. Our proof goes along the same line as the original one by
Y. Lee (\cite{yl}), but we need to modify several points. 

Before the proof, we point out that we might prove Theorem \ref{yl} via Theorem \ref{red.technique} as in the previous section,
provided that the following conjecture would be true (below $W$ is the ring of Witt vectors and $K,k$ are the field of fractions
and the residue field of W, respectively):

\begin{conj}\label{lsc}
Let $X_W\to \spec W$ be a smooth proper morphism where $X_W$ is an integral scheme.
Let $D_W$ be an effective $\mathbb{R}$-divisor on $X_W$, such that no irreducible component
is contained in a fiber of the projection to $\spec W$. By $X_K$ and $D_K$ we denote the 
restrictions of $X_W$ and $D_W$ over the generic point of $\spec W$. Similarly $X_k,D_k$
denote the restrictions of $X_W$ and $D_W$ over the closed point of $\spec W$.
Then if $(X_k,D_k)$ is log canonical, so is $(X_K,D_K)$.
In particular $\lct(X_k,D_k)\le\lct(X_K,D_K)$.
\end{conj}

Note that the above conjecture can be proven when  $(X_k,D_k)$ has a good log resolution
(here "good" means that it is isomorphic outside of the support of $D_k$),
following \cite{mus}. In \cite{mus},
the lower semi-continuity of log canonical thresholds in a family of projective log pairs with non-singular
ambient varieties is proven
when the base scheme of the family is defined in characteristic zero. We need the last assumption
because the existence of good log resolution is not yet
established in positive characteristics in full generality. Today basic results on
motivic integrations are established over arbitrary perfect fields (see \cite{y}), so the arguments in
\cite{mus} can be applied to our case without change under the existence of good log resolutions.

%%%%%%%%%%%%%%%%%%%%%%%Log canonicity in positive characteristics

\subsection{Log canonicity in positive characteristics}

In this subsection, we discuss how the log canonicity of log pairs are preserved under finite
morphisms. Some properties of log canonicity which hold in characteristic zero fail in
characteristic $p>0$, but we can circumvent those difficulties and obtain Proposition \ref{kl},
which is the key for the proof Theorem \ref{yl}.

When the characteristic of the base field is zero, it is well known that the log canonicity is preserved
under finite dominant morphisms (see \cite[Proposition 5.20(4)]{km}). Namely:

\begin{theo}\label{lc in char zero}
Let $g:X'\to X$ be a finite dominant morphism of normal varieties over a field of characteristic zero.
Let $\Delta$ (resp. $\Delta^{'}$) be a $\mathbb{Q}$-divisor on $X$ (resp. $X'$) such that
$K_X+\Delta$ is $\mathbb{Q}$-Cartier and $g^{*}(K_X+\Delta)=K_{X'}+\Delta^{'}$. Then
$(X,\Delta)$ is log canonical if and only if $(X',\Delta')$ is.
\end{theo}
We should note that the canonical divisors $K_X$ and $K_{X'}$ in Theorem
\ref{lc in char zero} above are chosen
in such a way that $K_{X'}=g^{*}K_X+R$ holds, where $R$ is the ramification divisor of $g$. 

When the characteristic of the base field is positive, we need to modify Theorem \ref{lc in char zero}.
First we consider the case when $g$ is separable. In this case we may have wild ramifications, so
we only have a weaker version of the ramification formula:

\begin{lem}
Let $g:X\to Y$ be a finite separable
morphism between normal varieties over $k$. Let $E\subset X$ be a prime divisor on $X$ and $r$ be the
ramification index of $g$ along $E$. Then there exists a non-negative integer $b\ge r-1$ such that
$K_X=g^{*}K_Y+bE$ holds around the generic point of $E$.
\end{lem}

\begin{proof}
Set $V=Y\setminus \sing Y$ and $U=g^{-1}(V)-\sing X$. Note that the closed subsets we have through
away have codimension greater than $1$.
Over $U$ we have the following exact sequence:
\begin{equation}\label{1es}
0\to g^{*}\Omega_{V}\xrightarrow[]{f}\Omega_{U}\to\Omega_{U/V}\to 0.
\end{equation}
Since $g$ is separable, $\Omega_{U/V}$ generically vanishes (see \cite[Theorem 59]{mat}).
Hence $f$ is generically isomorphic. 

Let $F:g^{*}\mathcal{O}_V(K_V)\to\mathcal{O}_U(K_U)$ be the highest exterior product of
the morphism $f$ in (\ref{1es}) above. This is also generically isomorphic.
Therefore $\ker{F}$ is a torsion subsheaf of the torsion free sheaf
$g^{*}\mathcal{O}_V(K_V)$, so is trivial. Hence we see that $F$ is injective.

Take a generic closed point $e$ of $E\cap U$ which is contained in no other irreducible component
of $\supp{\Omega_{U/V}}$ except for $E$. Set $e'=g(e)$, $E'=g(E)$.
Choose systems of local coordinates $x_1,\dots,x_n$ at $e$ and $y_1,\dots,y_n$ at $e'$, satisfying
the following conditions:
\begin{enumerate}[(a)]
\item $E=\divis(x_1)$ near $e$ (resp. $E'=\divis(y_1)$).
\item $g^{*}y_i=x_i$ holds for all $i=2,\dots,n$.
\item there exists an invertible function $u$ at $e$ such that $g^{*}y_1=u\cdot x_1^{r}$.
\end{enumerate}
In (c), $r$ denotes the ramification index of $g$ along $E$. Now

\begin{eqnarray}
F(g^{*}(dy_1\wedge\dots\wedge dy_n)) &=& d(u\cdot x_1^r)\wedge dx_2\wedge dx_3
\wedge\dots\wedge dx_n \nonumber\\
&=& \left(\frac{\partial u}{\partial x_1}x_1+ru\right)x_1^{r-1}dx_1\wedge dx_2\dots
\wedge dx_n\label{ramif}, \\\nonumber
\end{eqnarray}

hence there exists some non-negative integer $b$ such that
\begin{eqnarray*}
K_X=g^{*}K_Y+bE \\
\end{eqnarray*}
holds in a neighborhood of $e$.

If $r\not\equiv 0 \pmod{p}$, $b=r-1$. Now assume that
$E$ is wildly ramifying. Then 
\begin{eqnarray*}
(\ref{ramif}) & =& \frac{\partial u}{\partial x_1}x_1^{r}dx_1\wedge dx_2\dots
\wedge dx_n \\
& \not=& 0, \\
\end{eqnarray*}
since otherwise $F$ is not generically isomorphic.
In this case we see that
\begin{equation*}
b=\mathop{\mathrm{val}}\nolimits_{E}\left(\frac{\partial u}{\partial x_1}\right)+r\ge r,
\end{equation*}
where $\mathop{\mathrm{val}}_{E}$ denotes the valuation corresponding to $E$.

\end{proof}

\begin{rem}
Ramification formula for inseparable morphisms are discussed in \cite{rs}. In this case
the ramification divisor is defined only up to linear equivalence. If we adopt this version
of ramification formula, the `only if' part of Theorem \ref{lc in char zero} does not
hold in general. For our purpose we need not to deal with inseparable cases. 
\end{rem}

With the weaker version of ramification formula above, we can prove that
the 'only if' part of the Theorem \ref{lc in char zero} still holds 
for separable morphisms:

%%%%%%%%%%%%%%%%%%%%%%pull-back

\begin{prop}\label{pull-back}
Let $k$ be an algebraically closed field of characteristic $p>0$. Let $g:X'\to X$ be a finite separable
morphism of normal varieties over $k$. Let $\Delta$ (resp. $\Delta^{'}$) be a $\mathbb{Q}$-divisor on $X$ (resp. $X'$) such that
$K_X+\Delta$ is $\mathbb{Q}$-Cartier and $g^{*}(K_X+\Delta)=K_{X'}+\Delta^{'}$. Then
if $(X,\Delta)$ is log canonical, so is $(X',\Delta')$.
\end{prop}
\begin{proof}
The proof goes along the same line as the proof of (\cite[Prop 5.20(4)]{km}), once we replace
the ramification formula by the weaker version given above.
\end{proof}

\begin{rem}
In general, the 'if' part of Theorem \ref{lc in char zero} holds only when
there exists no wildly ramifying divisor. In such a case, the proof
goes as in characteristic zero. If some of the ramification divisors are wildly ramifying
it may not hold. An example is:

\begin{exmp}
Let $X=X'=\mathbb{A}^{1}_{k}$. Set $g:X'\to X; g(x)=x^p(x+1)$, $\Delta=\frac{p+1}{p}\divis (x)$
and $\Delta'=\divis (x)$. Since $g^{*}dx=x^pdx$, we obtain
\begin{equation*}
g^{*}(K_X+\Delta)=K_{X'}+\Delta'.
\end{equation*}
Note that $(X,\Delta)$ is not lc, but $(X',\Delta')$ is. 
\end{exmp}

\end{rem}

Using Proposition \ref{pull-back}, we can extend
\cite[Proposition 8.13]{sp} over arbitrary fields:

%%%%%%%%%%%%%%%%%%%%%%%%kl

\begin{prop}\label{kl}
Take any $f\in k[x_1,\dots,x_n]$. Assign a weight $w=(w(x_i))_{i=1,\dots,n}\in{(\mathbb{Z}_{\ge 0})}^n
\setminus\{0\}$
to the variables $x_1,\dots,x_n$ and let
$w(f)$ be the weighted multiplicity of $f$ (= the lowest weight of the monomials occurring in $f$).
Then

\begin{equation*}
\frac{1}{\lct_{0}(\mathbb{A}^n,\divis(f))}\ge\frac{w(f)}{\sum_{i=1}^{n}w(x_i)}.
\end{equation*}
\end{prop}

\begin{proof}
%%%%%%%%%%%
\begin{Step}\label{separable case}%1
First we establish the inequality for those $w$'s such that
$w(x_i)>0$ holds for all $i=1,\dots,n$, and
$p$ divides none of the $w(x_i)$'s.

In this case the inequality can be established along the
same line as the original proof, since we have Proposition \ref{pull-back}. For the sake
of completeness, we reestablish the argument.

Consider $g:\mathbb{A}^{n}_{k}\to\mathbb{A}^{n}_{k}$ given by $g(x_i)=x_i^{w(x_i)}$.
By the assumptions on $w(x_i)$'s, $g$ is dominant and separable.
Take a real number $c\in\mathbb{R}_{\ge 0}$ and assume $(\mathbb{A}^{n}_{k},c\cdot
\divis(f))$ is lc at $0$. Now calculate the pull-back of $K_{\mathbb{A}^{n}_{k}}+
c\cdot\divis(f)$ by $g$:
\begin{eqnarray*}
 &&g^{*}(K_{\mathbb{A}^{n}_{k}}+c\cdot\divis(f)) \\
&=&K_{\mathbb{A}^{n}_{k}}
+\sum_{i=1}^{n}(1-w(x_i))\divis(x_i)+c\cdot\divis(f(x_1^{w(x_1)},\dots,x_n^{w(x_n)})) \\
&=:&K_{\mathbb{A}^{n}_{k}}+\Delta'. \\
\end{eqnarray*}

By Proposition \ref{pull-back}, we see $(\mathbb{A}^{n}_{k},\Delta')$ is lc at $0$.
Let $E$ be the exceptional divisor of the blow-up of $\mathbb{A}^{n}_{k}$ at the origin.
We know that $a(E;\mathbb{A}^{n}_{k},\Delta')\ge-1$ holds. With a calculation we see that
$a(E;\mathbb{A}^{n}_{k},\Delta')$ equals to
$-1+\sum_{i}w(x_i)-cw(f)$, obtaining the inequality.
\end{Step}
%%%%%%%%%%

\begin{Step}\label{general case}%2
Now consider the continuous function $\varphi:{(\mathbb{Q}_{\ge 0})}^n\setminus\{0\}\to \mathbb{Q}$
defined by
\begin{equation*}
\varphi{(w)}=\frac{w(f)}{\sum_iw(x_i)},
\end{equation*}
as in the case when $w(x_i)$'s are integers.
If we replace $w$ by some positive multiple of it, the value of $\varphi$ never changes. Therefore $\varphi$
factors through the quotient space 
\begin{equation*}
S:={(\mathbb{Q}_{\ge 0})}^n\setminus\{0\}/\mathbb{Q}_{>0},
\end{equation*}
inducing the continuous function 
$\overline{\varphi}:S\to \mathbb{Q}$.

The set of points represented by those $w$'s satisfying the assumptions in Step \ref{separable case}
is dense in $S$. Hence, by the continuity of $\overline{\varphi}$, we see that
\begin{equation*}
\overline{\varphi}(s)\le\frac{1}{\lct_{0}(\mathbb{A}^n,\divis(f))}
\end{equation*}
holds for arbitrary $s\in S$. We finish the proof.

\end{Step}
\end{proof}

\begin{rem}
Step \ref{general case} in the proof above is inevitable, for $(\mathbb{A}^n_{k},\Delta')$ need not be lc if
$g$ is inseparable.
For example, consider the case $n=2$,
$w(x_1)=w(x_2)=p$ and $f(x_1,x_2)=x_1-x_2$. In this case $\lct_{0}(\mathbb{A}^2_{k},
\divis(f))=1$.
On the other hand 
\begin{equation*}
\Delta'=(1-p)\divis(x_1x_2)+p\cdot\divis(x_1-x_2),
\end{equation*}
hence $(\mathbb{A}^2_{k},\Delta')$ is not lc at the origin. Note also that even in this case
$a(E;\mathbb{A}^2_{k},\Delta')\ge -1$ holds, because of Proposition \ref{kl}.

\end{rem}

%%%%%%%%%%%%%%%%%%%%%%%%%%%%%%%%Proof of Theorem \ref{yl}

\subsection{Proof of Theorem \ref{yl}}

\begin{proof}[Proof of Theorem \ref{yl}]
We only discuss stable case. Semi-stable case can be proven exactly in the same way.
We have only to confirm the inequality (\ref{eq of nc by yl}) of Lemma \ref{nc by yl}. 
On the other hand, by the assumption and Proposition \ref{kl},
the inequality clearly holds.
\end{proof}

\begin{rem}
Theorem \ref{yl} has the following direct corollary, which is slightly weaker:
\begin{cor}\label{Fpt version}
If $\Fpt(\mathbb{G},Z(X))>\frac{n+1}{d}$
(resp. $\ge\frac{n+1}{d}$), then $X$ is Chow stable (resp. Chow semi-stable).
\end{cor}
Above $\Fpt$ denotes the \textit{F-pure threshold} of the pair $(\mathbb{G},Z(X))$.

Corollary \ref{Fpt version} is deduced from Theorem \ref{yl} via
\cite[Theorem 3.3]{hw}, which says:
\begin{quote}
F-pure $\Rightarrow$ log canonical.
\end{quote}
We can also show Corollary \ref{Fpt version} by directly proving
the $\Fpt$ version of Proposition \ref{kl}, using the Fedder-type
criterion for F-purity due to \cite{hw}.
\end{rem}

%%%%%%%%%%%%%%%%%%%%%%%%%%%%%
%%%%%%%%%%%%%%%%%%
%%%%%%%%%%%%%%%%%%%%%%%%%%
%%%%%%%%%%%%%%%%%%
%%%%%%%%%%%%%%%%%%%%%%%%%%%
%%%%%%%%%%%%%%%%%%%
%%%%%%%%%%%%%%%%%%%%%%%%%%%%%%%%%%%%%%%%%%%%
%%%%%%%%%%%%%%%%%
%%%%%%%%%%%%%%%%%%%%%%%%%%%%%%%%%%%
%%%%%%%%%%%%%%%%%Appendix

\appendix

%%%%%%%%%%%%%%%%%%%%%%%%%%%Chow stability of the sum

 \section{Chow stability of the sum}\label{sum}

In this section we show that the sum of two Chow semi-stable cycles of the same
dimension are again Chow semi-stable. Moreover if one of them is stable,
it follows that the sum also becomes stable.

\begin{prop}\label{stability of sum}
Let $Y,Z$ be Chow semi-stable cycles of the same dimension in a projective
space $\mathbb{P}^{n}_{k}$. Then $Y+Z$ is again Chow semi-stable.
Furthermore if $Y$ is Chow stable, so is $Y+Z$.
\end{prop}

In the proof  we freely use the notation like $\lim_{t\to 0}\lambda(t)\cdot F$, as in
\cite{git}, since the idea becomes clearer. To be
logically complete, we of course need to replace the argument suitably. It is a routine work,
so we omit the detail.

\begin{proof}
Let $d$,$e$ be the degrees of $Y$ and $Z$ respectively. 
Let $F\in\mathcal{B}_d$, $G\in\mathcal{B}_e$ be the Chow forms of $Y$ and $Z$, 
respectively. Then the Chow form of $Y+Z$ is given by $F\cdot G\in\mathcal{B}_{d+e}$.

Choose a non-trivial $1$-parameter subgroup (1-PS) 
$\lambda:\mathbb{G}_{m}\to SL(n+1,k)$. Via $\lambda$ we pull back
the canonical actions of $SL(n+1,k)$
onto $\mathcal{B}_d,\mathcal{B}_e$ and $\mathcal{B}_{d+e}$ to $\mathbb{G}_{m}$.
Now consider the natural multiplication map
$\mu:\mathcal{B}_d\times\mathcal{B}_e\to\mathcal{B}_{d+e}$, given by $(F,G)\mapsto F\cdot G$.
If we pose the diagonal action of $\mathbb{G}_{m}$ on the source, $\mu$ becomes
equivariant.

Assume that $Y,Z$ are both Chow semi-stable. Then both $\lim_{t\to 0}\lambda(t)\cdot F\not= 0$
and $\lim_{t\to 0}\lambda(t)\cdot G\not= 0$ holds. Now since we know that $\mu$ is continuous,
\begin{equation*}
\begin{split}
& \lim_{t\to 0}\lambda(t)\cdot(F\cdot G) \\
= & (\lim_{t\to 0}\lambda(t)\cdot F)\cdot
(\lim_{t\to 0}\lambda(t)\cdot G) \\
\not=& 0, \\
\end{split}
\end{equation*}
 since $\mathcal{B}=\oplus_{d\ge 0}\mathcal{B}_d$
is an integral domain. Therefore $Y+Z$ is again Chow semi-stable.
 
Second, assume further that $Y$ is Chow stable. Then
$\lim_{t\to 0}\lambda(t)\cdot F=\infty$, so that 
\begin{equation*}
\begin{split}
& \lim_{t\to 0}\lambda(t)\cdot(F\cdot G) \\
= & (\lim_{t\to 0}\lambda(t)\cdot F)\cdot
(\lim_{t\to 0}\lambda(t)\cdot G) \\
= & \infty\cdot(\lim_{t\to 0}\lambda(t)\cdot G) \\
= & \infty, \\
\end{split}
\end{equation*}
since $(\lim_{t\to 0}\lambda(t)\cdot G)$ is not $0$.
Therefore $Y+Z$ is Chow stable.

\end{proof}

\begin{rem}
We can not expect the converse of Proposition \ref{stability of sum} at all.
There exists a semi-stable cycle such that all of its subcycles are unstable: 

\begin{exmp}
Take the union of three lines on a plane which are in a general position. The union
itself is Chow semi-stable (see \cite[\S 4-2]{git}), but lines and reducible conics
 on a plane are Chow unstable.
\end{exmp}
\end{rem}

%%%%%%%%%%%%%%%%%%%%%Prop: multiple case
However, the following holds:

\begin{prop}
Let $Z$ be a cycle of $\mathbb{P}^n_{k}$. Then the followings are equivalent:
\begin{enumerate}
\item $Z$ is Chow (semi-)stable.
\item $mZ$ is Chow (semi-)stable for any positive integer $m\in\mathbb{Z}_{>0}$.
\item $mZ$ is Chow (semi-)stable for some positive integer $m\in\mathbb{Z}_{>0}$.
\end{enumerate}
\end{prop}
\begin{proof}
We have only to prove (3)$\Rightarrow$(1).
Let $G$ be the Chow form of $Z$. Then $G^{m}$ gives the Chow form of $mZ$.
Assume that $mZ$ is Chow semi-stable. Take
any 1-PS $\lambda$ as in the proof of the Proposition \ref{stability of sum}. Then
$0\not=\lim_{t\to 0}\lambda(t)\cdot G^m=\left(\lim_{t\to 0}\lambda(t)\cdot G\right)^{m}$, hence
$\lim_{t\to 0}\lambda(t)\cdot G\not=0$. Therefore $Z$ is semi-stable. Stable case can also be
shown via a similar argument.
\end{proof}

\begin{exmp}\label{counter example}
Let $Y\subset\mathbb{P}^n_{k}$ be a non-singular hypersurface of degree three, which is
Chow stable by Theorem \ref{hp}. By Proposition \ref{stability of sum},
$mY$ is also Chow stable for all the positive integers $m$. 
On the other hand, $\lct(\mathbb{P}^n_{k},mY)=\frac{1}{m}$ and hence
$\frac{1}{m}<\frac{n+1}{3m}$ if $n\ge 3$. Thus we obtain a sequence of examples of
Chow stable hypersurfaces whose stability can not be detected by Theorem \ref{yl}.
\end{exmp}

%%%%%%%%%%%%%%%%%%%%%%%%%%%%%%%%%
%%%%%%%
%%%%%%%%%%%%%%%%%%%%%%%
%%%%%%%%
%%%%%%%%
%%%%%%%%
%%%%%%%%%%%%%%%%%%%%%Bibliography%%%%%%


\begin{thebibliography}{EGA IV-3}

\bibitem[D]{dol} I. Dolgachev, \textit{Lectures on invariant theory},
London Mathematical Society Lecture Note Series, 296, Cambridge University Press, Cambridge, 2003.

\bibitem[EGA IV]{ega} A. Grothendieck, \textit{\'El\'ements de G\'eom\'etrie Alg\'ebrique IV,
\'Etude locale des sch\'emas et des morphismes de sch\'emas III}, Inst. Hautes \'Etudes Sci. Publ. Math. No. 28,
1966. 


\bibitem[GKZ]{gkz} I. M. Gelfand, M. M. Kapranov, and A. V. Zelevinsky, \textit{Discriminants, resultants,
and multidimensional determinants}. Mathematics: Theory and Applications. Birkh\"{a}user Boston, Inc., Boston,
MA, 1994.

\bibitem[GIT]{git} D. Mumford, J. Fogarty, and F. Kirwan, \textit{Geometric invariant theory,
Third edition}, Ergebnisse
der Mathematik und ihrer Grenzgebiete (2), vol. 34, Springer-Verlag, Berlin, 1994.

\bibitem[HW]{hw} N. Hara and K.-i. Watanabe, \textit{F-regular and F-pure rings vs. 
log terminal and log canonical singularities}, J. Alg. Geom. \textbf{11} (2002), 363--392.

\bibitem[Ha]{ha} R. Hartshorne, \textit{Algebraic geometry}, Graduate Texts in Mathematics, 
No. 52, Springer-Verlag, New York-Heidelberg, 1977.

\bibitem[Ko]{sp} J. Koll\'ar, \textit{Singularities of pairs}, Algebraic geometry---Santa Cruz 1995, 221--287, 
Proc. Sympos. Pure Math., \textbf{62}, Part 1, Amer. Math. Soc., 1997. 

\bibitem[Ko2]{ko2} J. Koll\'ar, \textit{Rational curves on algebraic varieties},
Ergebnisse der Mathematik und ihrer Grenzgebiete. 3. Folge. A Series of Modern Surveys in Mathematics,
32. Springer-Verlag, Berlin, 1996. 

\bibitem[KoM]{km} J. Koll{\'a}r and S. Mori, \textit{Birational geometry of algebraic varieties},
Cambridge Tracts in Mathematics, 134 (1998).

\bibitem[KL]{kl} H. Kim and Y. Lee, \textit{Log canonical thresholds of semistable plane curves},
 Math. Proc. Camb. Phil. Soc. \textbf{137} (2004), 273--280.

\bibitem[L]{rl} R. Laza, \textit{The moduli space of cubic fourfolds}, J. Alg. Geom. \textbf{18} (2009), no. 3, 511--545.

\bibitem[Laz]{laz} R. Lazarsfeld, \textit{Positivity in algebraic geometry II}, Ergebnisse der Mathematik und
ihrer Grenzgebiete. 3. Folge. A Series of Modern Surveys in Mathematics, 48. Springer-Verlag, Berlin, 2004. 

\bibitem[Le]{yl} Y. Lee, \textit{Chow stability criterion in terms of log canonical threshold}, J. Korean Math. Soc.
\textbf{45} (2008), no.2, 467--477.

\bibitem[M]{mat} H. Matsumura, \textit{Commutative algebra, Second edition}, Mathematics Lecture Note Series,
56 Benjamin/Cummings Publishing Co., Inc., Reading, Mass.


\bibitem[MM]{mm} H. Matsumura and P. Monsky, \textit{On the automorphisms of hypersurfaces},
J. Math. Kyoto Univ. \textbf{3} (1963/1964), 347--361.

\bibitem[Mu]{mukai} S. Mukai, \textit{An introduction to invariants and moduli},
Cambridge Studies in Advanced Mathematics, 81 (2003).

\bibitem[Mus]{mus} M. Musta\c{t}\u{a}, \textit{Singularities of pairs via jet schemes},
J. Amer. Math. Soc. \textbf{15} (2002), no. 3, 599--615.

\bibitem[N]{n} L. Ness, \textit{Mumford's numerical function and stable projective hypersurfaces},
in Algebraic geometry (Proceedings of summer meeting at university of Copenhagen, Copenhagen)
(1978), 417--453.

\bibitem[RS]{rs} A. N. Rudakov and I. R. \v{S}afarevi\v{c}, \textit{Inseparable morphisms of
algebraic surfaces}, Math. USSR Izv. \textbf{10} (1976), 1205--1237.

\bibitem[S]{s} J. P. Serre, \textit{Local fields},  Graduate Texts in Mathematics, No. 67, Springer-Verlag,
New York-Berlin, 1979.

\bibitem[Se]{sesh} C. S. Seshadri,
\textit{Geometric reductivity over arbitrary base}, Advances in Math. \textbf{26} (1977), no. 3, 225--274.

\bibitem[Sh]{shah} J. Shah, \textit
{A complete moduli space for $K3$ surfaces of degree $2$}, Ann. of Math. (2) \textbf{112} (1980),
no. 3, 485--510.

\bibitem[Y]{y} T. Yasuda, \textit{Twisted jets, motivic measures and orbifold cohomology},
Compos. Math. \textbf{140} (2004), no. 2, 396--422.

\end{thebibliography}
 \end{document}